\documentclass{amsart}
\pdfoutput=1

\usepackage[utf8]{inputenc}
\usepackage{amsfonts,mathtools,amsmath,amsthm,amssymb,tikz,color,enumitem,latexsym,tikz-cd,mathrsfs,graphicx,verbatim,MnSymbol}
\usepackage[algoruled,linesnumbered]{algorithm2e}
\usepackage[breaklinks]{hyperref}
\usepackage{thmtools}
\usepackage{thm-restate}
\usepackage[nameinlink]{cleveref}
\usepackage{ogonek}
\usepackage[center]{caption}

\usepackage[backref=true]{biblatex}
\bibliography{sources.bib}
\allowdisplaybreaks

\hypersetup{
	colorlinks,
	hyperfootnotes=true,
	linkcolor={red!50!black},
	citecolor={blue!50!black},
	urlcolor={blue!80!black}
}

\DefineBibliographyStrings{english}{%
	backrefpage = {page},%
	backrefpages = {pages},%
}

\usetikzlibrary{arrows}

\tikzset{slopearrow/.style={sloped, anchor=south}}

\declaretheorem[name=Theorem,refname={theorem,theorems},
Refname={Theorem,Theorems},numberwithin=section]{thm}
\declaretheorem[name=Proposition,refname={proposition,propositions},
Refname={Proposition,Propositions},sibling=thm]{prop}

\declaretheorem[name=Lemma,refname={lemma,lemmas},
Refname={Lemma,Lemmas},sibling=thm]{lem}

\declaretheorem[name=Conjecture,refname={conjecture,conjectures},
Refname={Conjecture,Conjectures},sibling=thm]{conj}

\def\CC{\mathbb C}

\def\RR{\mathbb R}
\def\ZZ{\mathbb Z}
\def\QQ{\mathbb Q}
\def\FF{\mathbb F}
\def\PP{\mathbb P}

\def\k{k}

\DeclareMathOperator\rank{rank}

\DeclareMathOperator\ord{ord}

\DeclareFontFamily{U}{wncy}{}
\DeclareFontShape{U}{wncy}{m}{n}{<->wncyr10}{}
\DeclareSymbolFont{mcy}{U}{wncy}{m}{n}
\DeclareMathSymbol{\Sh}{\mathord}{mcy}{"58}

\title{Root numbers of a family of elliptic curves and two applications}
\author{Jonathan Love}
\address[Jonathan Love]{McGill University}
\email{jon.love@mcgill.ca}
\thanks{Supported by CRM-ISM postdoctoral fellowship}
\date{October 2023}

\begin{document}
	
	\begin{abstract}
		For each $t\in\QQ\setminus\{-1,0,1\}$, define an elliptic curve over $\QQ$ by
		\begin{align*}
			E_t:y^2=x(x+1)(x+t^2).
		\end{align*}
		Using a formula for the root number $W(E_t)$ as a function of $t$ and assuming some standard conjectures about ranks of elliptic curves, we determine (up to a set of density zero) the set of isomorphism classes of elliptic curves $E/\QQ$ whose Mordell-Weil group contains $\ZZ\times \ZZ/2\ZZ\times \ZZ/4\ZZ$, and the set of rational numbers that can be written as a product of the slopes of two rational right triangles.
	\end{abstract}

	\maketitle
	
	\section{Introduction}
	
	Let $E$ be an elliptic curve over $\QQ$. By a theorem of Mordell (see for example~\cite[Theorem VIII.4.1]{silverman}), the group $E(\QQ)$ of rational points is isomorphic to $E(\QQ)_{\text{tors}}\times \ZZ^{r}$, where $E(\QQ)_{\text{tors}}$ is a finite abelian group and $r=\rank E(\QQ)$ is a non-negative integer known as the \emph{rank} of $E(\QQ)$. While the torsion subgroup $E(\QQ)_{\text{tors}}$ is straightforward to compute~\cite[Corollary VIII.7.2]{silverman}, the rank of $E(\QQ)$ is considerably more mysterious. There are no algorithms that are guaranteed to successfully compute the rank of a given elliptic curve over $\QQ$, and there are many unproven conjectures and open problems relating to the distribution and properties of ranks of elliptic curves; see~\cite{rubinsilverberg} for a survey.
	
	Given $t\in\QQ$, define the curve $E_t$ over $\QQ$ by
	\begin{align}\label{eq:Edef}
		E_t:y^2=x(x+1)(x+t^2).
	\end{align} 
	This is an elliptic curve for all $t\in\QQ\setminus\{-1,0,1\}$, and this family of curves appears in several parametrization problems (see \cref{sec:apps}). 	For $x\in\QQ^\times$, let $\ord_p(x)$ denote the integer $n$ such that $x=p^n\frac{a}{b}$ for integers $a,b$ relatively prime to $p$. Given a set $S\subseteq\QQ$, define the \emph{density} of $S$ to be
	\[\lim_{X\to\infty} \frac{|\{\frac ab\in\QQ:|a|,|b|\leq X\}\cap S|}{|\{\frac ab\in\QQ:|a|,|b|\leq X\}|},\]
	if the limit exists.
	
	\begin{thm}\label{thm:rootnumtotest}
		Given $t\in\QQ\setminus\{0,\pm1\}$, let $\mathcal{P}_t$ denote the set consisting of odd primes $p$ with $\ord_p(t)\neq 0$, primes $p\equiv 1\mod 4$ with $\ord_p(t^2-1)>0$, and $2$ if and only if neither $\ord_2(t)=\pm2$ nor $\ord_2(t^2-1)=3$. Set 
		\[\mathcal{T}=\{t\in\QQ\setminus\{0,\pm1\}:|\mathcal{P}_t|\text{ even}\}\cup S.\]
		\begin{itemize}
			\item The set $\mathcal{T}$ has density $\frac12$.
			\item If \cref{conj:parity} holds, then $t\in \mathcal{T}$ implies $|E_t(\QQ)|=\infty$. 
			\item If \cref{conj:parity,conj:density} both hold, then there is a set $S\subseteq \QQ$ of density $0$ such that $|E_t(\QQ)|=\infty$ implies $t\in S\cup \mathcal{T}$.
		\end{itemize}
	\end{thm}
	
	This provides an easily computable criterion which, assuming standard conjectures regarding ranks of elliptic curves, allows us to check whether $E_t(\QQ)$ has positive rank.
	
	We now discuss the conjectures assumed in \cref{thm:rootnumtotest}. To each elliptic curve $E/\QQ$ we can assign a \emph{root number} $W(E)\in\{-1,1\}$, which appears as the sign in the functional equation of the $L$-function of $E$ (see \cref{sec:background} for more details). A consequence of the Birch and Swinnerton-Dyer conjecture~\cite{bsd} is the following.
	\begin{conj}\label{conj:parity}
		For any elliptic curve $E/\QQ$, we have	$W(E)=(-1)^{\rank E(\QQ)}$.
	\end{conj}
	\noindent
	Assuming this conjecture, the root number can often be used as a stand-in for the rank. For instance, if $W(E)=-1$, then $\rank E(\QQ)$ is odd, which implies $E(\QQ)$ is infinite. We present a formula for the root number $W(E_t)$ (\cref{lem:rootnum}), and apply this description to two diophantine problems. An equivalent form of this root number formula was previously computed by Gonz\'alez-Jim\'enez and Xarles \cite[Proposition 11]{xarles}; we include a derivation in this paper for completeness.

	Given an elliptic curve $\mathcal{E}$ over a function field $\QQ(T)$, and any $t\in\QQ$, let $\mathcal{E}_t$ denote the specialization of $\mathcal{E}$ at $T=t$; $\mathcal{E}_t$ will be an elliptic curve for all but finitely many $t\in\QQ$.
	
	\begin{conj}\label{conj:density}
		For any elliptic curve $\mathcal{E}/\QQ(T)$ with nonconstant $j$-invariant, the set
		\[\left\{t\in\QQ:\rank \mathcal{E}_t(\QQ)\geq 2+\rank \mathcal{E}(\QQ(T))\right\}.\]
		has density zero.	
	\end{conj}
	See for instance \cite[Appendix A]{conradconrad} or \cite[Remark 3.1]{miller} for a discussion of this conjecture. 

	\subsection{Applications}\label{sec:apps}

	We introduce two interpretations of the curves $E_t$, each giving an application of \cref{thm:rootnumtotest}. First, considered as an elliptic curve over the function field $\QQ(T)$, \cref{eq:Edef} defines a model for the universal elliptic curve with $\ZZ/2\ZZ\times \ZZ/4\ZZ$ torsion subgroup.
	\begin{prop}
		Assume \cref{conj:parity}. If $t\in\mathcal{T}$, then 
		\[\ZZ\times\ZZ/2\ZZ\times\ZZ/4\ZZ\leq E_t(\QQ).\]
	\end{prop}
	\begin{proof}
		We have $\ZZ/2\ZZ\times\ZZ/4\ZZ\leq E(\QQ)$ if and only if $E\cong E_t$ for some $t\in\QQ\setminus\{0,\pm 1\}$~, and \cref{thm:rootnumtotest} gives us a (conjectural) criterion under which $E_t(\QQ)$ has positive rank.
	\end{proof}
	
	Secondly, the curves $E_t$ parameterize solutions to a Diophantine problem with a simple geometric interpretation. Fix $t\in\QQ_{>0}$ and consider the system
	\begin{align}\label{eq:righttri}
		\left\{\begin{aligned}
			1+a^2&=b^2,\\
			t^2+a^2&=c^2,
		\end{aligned}\right.\qquad a,b,c\in\QQ_{>0}.
	\end{align}
	\begin{figure}
		\begin{center}
			\begin{tikzpicture}
				
				\newcommand\X{1.5}
				\newcommand\Y{2}
				\newcommand\side{3.5}
				
				\draw[dashed]  (\X,\side) edge (\X,\Y);
				\draw  (0,\Y) edge (\X,\Y);
				\draw  (\side,\Y) edge (\X,\Y);
				\draw[dashed] (\X,\side) edge (0,\Y);
				\draw[dashed] (\X,\side) edge (\side,\Y);
				
				\node at (0.85,1.75) {$1$};
				\node at (1.75,2.55) {$a$};

				\node at (2.35,1.75) {$t$};
				\node at (0.5,2.95) {$b$};
				\node at (2.65,2.95) {$c$};
			\end{tikzpicture}
			\caption{Solutions of \cref{eq:righttri}.}\label{fig:righttri}
		\end{center}
	\end{figure}
	Solutions of this system correspond to pairs of rational right triangles with a common leg such that the ratio of the remaining legs equals $t$ (see \cref{fig:righttri}); equivalently, a solution exists if and only if $t$ can be expressed as a product $t=\frac{a}{1}\cdot\frac{t}{a}$ of slopes of rational right triangles. We will prove the following classification.
	\begin{prop}\label{prop:classify}
		Let $t\in\QQ_{>0}$. 
		\begin{itemize}
			\item If $|E_t(\QQ)|=\infty$, then \cref{eq:righttri} has infinitely many solutions.
			\item If $|E_t(\QQ)|\neq\infty$ and $t$ and $t+1$ are both squares in $\QQ$, then \cref{eq:righttri} has a single solution 
			\[(a,b,c)=(\sqrt{t},\,\sqrt{t+1},\,\sqrt{t(t+1)}).\]
			\item If $|E_t(\QQ)|\neq\infty$ and either $t$ or $t+1$ is not a square in $\QQ$, then \cref{eq:righttri} has no solutions.
		\end{itemize}
	\end{prop}
	The proof is given in \cref{sec:ratright}. As a consequence, \cref{thm:rootnumtotest} conjecturally classifies (up to a density zero set of exceptions) the set of rational numbers that can be written as a product of slopes of two rational right triangles in infinitely many ways.
	
	The family of curves $E_t$ has also been studied in other contexts; for instance, the subfamily
	\[y^2=x(x+1)\left(x+\left(\frac{1-u^2}{2u}\right)^2\right)\]
	was studied by Halbeisen and Hungerb\"uhler~\cite{halbeisen} in the context of solutions to \cref{eq:righttri} with the additional constraint that $1+t^2$ must also be a square. With this extra constraint, $1$, $t$, and $a$ form the sides of a box with rational edge lengths and rational face diagonals; hence \cref{thm:rootnumtotest} gives a leaky classifier for the possible ratios of side lengths in a rational cuboid (for a discussion of rational cuboids see~\cite{leech,luijk}). The author considers a more general family of curves and their relations to rational configurations in \cite{loveratdist}.
	
	\subsection{Comparison to congruent numbers}
		The system given by \cref{eq:righttri} bears a strong resemblance to the \emph{congruent number problem}, which fixes $n\in\ZZ_{>0}$ and asks for solutions to the system 
		\begin{align}\label{eq:congnum}
			\left\{\begin{aligned}
				a^2+b^2&=c^2,\\
				ab&=2n,
			\end{aligned}\right.\qquad a,b,c\in\QQ_{>0}.
		\end{align}
		In the same way that solutions to \cref{eq:righttri} correspond to rational points on $E_t$, solutions to \cref{eq:congnum} correspond to rational points on an elliptic curve $E^{(n)}:y^2=x^3-n^2x$. Further, in both cases there is a torsion subgroup of ``degenerate'' points that do not correspond to valid solutions of the original system. See for example the expositions~\cite{kconrad,chandrasekar} for more details on the congruent number problem.
		
		We highlight two key differences between the systems \cref{eq:righttri,eq:congnum} and their corresponding elliptic curves.
		\begin{itemize}
			\item One can prove that $E^{(n)}(\QQ)$ has no torsion points other than the degenerate points, so that \cref{eq:congnum} has a solution if and only if $E^{(n)}(\QQ)$ has positive rank. In contrast, it is possible for $E_t(\QQ)$ to have torsion points that yield valid solutions of \cref{eq:righttri}; see \cref{prop:extrators} for a classification of the values of $t$ for which this occurs.
			
			\item A formula due to Tunnell can be used to determine whether the analytic rank of $E^{(n)}$ is zero or positive~\cite{tunnell}, so by assuming the Birch and Swinnerton-Dyer conjecture, this gives a complete criterion that determines whether a given integer is a congruent number (not just up to a set of density zero). Tunnell's computation depends on the fact that the curves $E^{(n)}$ are all isomorphic over $\overline{\QQ}$. Specifically, each $E^{(n)}$ is a quadratic twist of the curve $y^2=x^3-x$, which allows Tunnell to apply a result due to Waldspurger~\cite{waldspurger} relating the central value of the $L$-function of an elliptic curve with that of each of its quadratic twists. The curves $E_t$, on the other hand, have different $j$-invariants, and so these techniques do not apply.
		\end{itemize} 
	
	\subsection{Outline}
	
	In \cref{sec:background}, we give more information about root numbers, and discuss prior work on computations of root numbers of elliptic curves in one-parameter families. We discuss the family of curves $E_t$ in \cref{sec:family} by defining a corresponding elliptic curve $\mathcal{E}$ over $\QQ(T)$, and compute the Mordell-Weil group $\mathcal{E}(\QQ(T))$. This is followed in \cref{sec:ratright} by a discussion of the Diophantine problem \cref{eq:righttri} and its relation to the family $\mathcal{E}$. Finally, the main results \cref{lem:rootnum,thm:rootnumtotest} are proved in \cref{sec:computeroot}.
	
	\subsection{Acknowledgments}
	The author would like to thank Michael Lipnowski, Andrew Granville, Dimitris Koukoulopoulos, and Xavier Xarles for several helpful conversations. During the writing of this note, the author was supported by a CRM-ISM Postdoctoral Fellowship.

	\section{Background: root numbers in families}\label{sec:background}
	An elliptic curve $E$ over $\QQ$ is a smooth genus one projective curve with a distinguished rational point. 
	Up to isomorphism, any such curve can be written as a (homogeneous) Weierstrass equation 
	\[y^2z+a_1xyz+a_3yz^2=x^3+a_2x^2z+a_4xz^2+a_6z^3,\]
	with $a_1,a_2,a_3,a_4,a_6\in\QQ$. The point $(x,y,z)=(0,1,0)$ is the only point with $z=0$; for convenience we often set $z=1$ to obtain an affine Weierstrass equation, and simply remember the existence of the additional point at infinity (labelled $O$). Define
	\begin{align*}
		b_2&=a_1^2+4a_2\\
		c_4&=b_2^2-24(2a_4+a_1a_3)\\
		c_6&=-b_2^3+36b_2(2a_4+a_1a_3)-216(a_3^2+4a_6)\\
		\Delta&=\frac{c_4^3-c_6^2}{1728}\qquad\qquad j=\frac{c_4^3}{\Delta}.
	\end{align*}
	The Weierstrass equation defines a nonsingular curve if and only if $\Delta\neq 0$. There are many different choices of coefficients that all yield curves isomorphic to $E$; we will refer to these as \emph{models} of $E$. If the coefficients $a_i$ are all integers, the Weierstrass equation defines an \emph{integral model} of $E$.
	
	Let $p$ be a prime, and let $E$ be given by an integral model. We recall how to compute the reduction type of $E$ at a prime $p$ (see for example~\cite[Section IV.9]{silvermanadv}). If $p\nmid\Delta$, then $E$ has \emph{good reduction} at $p$. If $p\mid\Delta$, then $\overline{E}$, the reduction of $E$ modulo $p$, has a singular point; replacing $E$ with an isomorphic Weierstrass model if necessary, we can assume the singularity of $\overline{E}$ is at $(0,0)$. Then if $b_2$ is a nonzero quadratic residue mod $p$, then $E$ has \emph{split multiplicative reduction} at $p$. If $b_2$ is not a quadratic residue mod $p$, then $E$ has \emph{non-split multiplicative reduction} at $p$. If $p\mid b_2$, then $E$ has \emph{additive reduction} at $p$.
	
	Now suppose $E$ is given by a \emph{minimal} integral model; that is, an integral model for which $|\Delta|$ is as small as possible. Set $a_p=p+1-|\overline{E}(\FF_p)|$. For $s\in\CC$, define 		
	\begin{align*}
		L_p(E,s):=\left\{\begin{array}{ll}
			(1-a_p p^{-s}+p^{1-2s})^{-1},&E\text{ has good reduction at }p;\\
			(1-p^{-s})^{-1},&E\text{ has split multiplicative reduction at }p;\\
			(1+p^{-s})^{-1},&E\text{ has non-split multiplicative reduction at }p;\\
			1,&E\text{ has additive reduction at }p.
		\end{array}\right.
	\end{align*}
	The $L$-function of $E$ is defined by
	\[L(E,s)=\prod_p L_p(E,s).\]
	This product converges for $s$ with real part at least $\frac32$. However, by the modularity theorem~\cite{modularity}, $L(E,s)$ has an analytic continuation to the entire complex plane, and satisfies the functional equation
	\begin{align}
		(2\pi)^{-s}\Gamma(s)N^{s/2}L(E,s)=W(E)(2\pi)^{s-2}\Gamma(2-s)N^{(2-s)/2}L(E,2-s),
	\end{align}
	where $N$ is the conductor of $E$ and $W(E)\in\{-1,1\}$ is the root number of $E$.
	
	Let $r\geq 0$ denote the order of vanishing of $L(E,s)$ at $s=1$; $r$ is called the \emph{analytic rank} of $E$. Then the functional equation can be written in the form
	\[f(s)(s-1)^r=W(E)f(2-s)(1-s)^r\]
	for some $f$ which is analytic and nonzero at $s=1$. Comparing the coefficient of $(s-1)^r$ on both sides, we must have $W(E)=(-1)^r$. The Birch and Swinnerton-Dyer conjecture~\cite{bsd} predicts that $r=\rank E(\QQ)$, which therefore implies \cref{conj:parity}. \Cref{conj:parity} has been proven for all $E$ such that the Tate-Shafarevich group of $E$ is finite~\cite{dokchister}, and has been proven unconditionally when $E$ has analytic rank $0$ or $1$, as a consequence of the Birch and Swinnerton-Dyer conjecture which has been proven in these cases~\cite{kolyvagin}.
	
	For each place $v$ of $\QQ$ ($v=p$ for $p$ a prime, or $v=\infty$), we can define a \emph{local root number} $W_p(E)$ using representation theory of the Weil-Deligne group of $E$ over $\QQ_p$ (originally defined in~\cite{deligne}; see also~\cite{weildeligne}). The local root number equals $1$ for all primes of good reduction for $E$, and can be computed explicitly for any given $E/\QQ$ and any prime $p$; this is classical for $p>3$ (see for example~\cite{rohrlich}), and the cases $p=2$ and $p=3$ were completed by Connell~\cite{connell} and Halberstadt~\cite{halberstadt}, with a reformulation of Halberstadt's results given by Rizzo~\cite{rizzo}. The (global) root number $W(E)$ is equal to the product of all local root numbers,\footnote{This is the original definition of $W(E)$~\cite{deligne}, and perhaps more natural than the definition of $W(E)$ as the sign in the functional equation, since the existence of the functional equation depends on the modularity theorem for curves over $\QQ$, and is not known for general number fields.} so unlike the rank, the root number of $E$ can be explicitly computed for any individual elliptic curve over $\QQ$. Algorithms to compute the root number have been implemented in computer algebra systems such as Magma.
	
	Despite the existence of efficient algorithms for computing the root number of an individual elliptic curve over $\QQ$, one cannot in general write down the root number in terms of a formula in the coefficients of an elliptic curve. However, if $\mathcal{E}$ is an elliptic curve over the function field $\QQ(T)$, it is typically possible (if tedious) to compute the root number $W(\mathcal{E}_t)$ of the specializations of $\mathcal{E}$ as a function of $t$. This has been done for many families of curves; some examples include the families $y^2=x^3-Dx$ and $y^3=x^3+A$ by Birch and Stephens~\cite{birchstephens}, and the family
	\[y^2+Txy=x^3-\frac{T(T-1)}{4}x^2-\frac{36T^2}{T-1728}x-\frac{T^3}{T-1728}\]
	(the $j$-invariant of $\mathcal{E}_t$ is equal to $t$) by Rohrlich~\cite{rohrlich}. Using Birch and Stephens' root number computation for $y^2=x^3-Dx$, Cassels and Schinzel showed that every specialization of the family $y^2=x(x^2-(7+7T^4)^2)$ has root number $-1$~\cite{casselsschinzel} (note that this family has constant $j$-invariant; it is believed that for families of curves over $\QQ$ with non-constant $j$-invariant, each root number occurs infinitely often~\cite[Appendix A]{conradconrad}).
	
	Define
	\[S^{\pm}:=\{t\in\QQ:\mathcal{E}_t\text{ is an elliptic curve and }W(\mathcal{E}_t)=\pm 1\}.\]
	In addition to explicit formulas defining these sets, much work has been done on statistical features of $S^{\pm}$: conditions under which they are nonempty, infinite, or dense in $\RR$, and the respective densities of $S^\pm$ in $\QQ$ or of $S^\pm\cap\ZZ$ in $\ZZ$. An unpublished paper of Helfgott shows that the average (in both senses) is nonzero for a large class of families, but also found families with non-constant $j$-invariant in which the average root number over $\QQ$ is nonzero~\cite{helfgott}. More recently, Desjardins proved infinitude of the sets $S^{\pm}$ for a larger class of families~\cite{desjardins}. Many other authors have also written on the sets $S^{\pm}$~\cite{manduchi,conradconrad,varilly,gouvea}.
	
	\section{The family of elliptic curves}\label{sec:family}
	
	Define the elliptic curve $\mathcal{E}$ over the function field $\QQ(T)$ by
	\begin{align}\label{eq:Ebigdef}
		\mathcal{E}:y^2=x(x+1)(x+T^2).
	\end{align} 
	This curve $\mathcal{E}$ has discriminant
	\begin{align}\label{eq:Edisc}
		\Delta(\mathcal{E})=16 t^4(t-1)^2(t+1)^2,
	\end{align}
	and for all $t\in\QQ\setminus\{0,\pm 1\}$, the specialization of $\mathcal{E}$ at $t$ is $E_t$ (\cref{eq:Edef}).
	
	\begin{lem}\label{lem:MWgroup}
		The Mordell-Weil group $\mathcal{E}(\QQ(T))$ is isomorphic to  $\ZZ/2\ZZ\times \ZZ/4\ZZ$.
	\end{lem}
	\begin{proof}
		We can exhibit a subgroup isomorphic to $\ZZ/2\ZZ\times \ZZ/4\ZZ$, consisting of the eight points
		\begin{align}\label{eq:torsionpoints}
			\begin{array}{cccc}
				O, & (-T,\,-T(T-1)) & (0,0) & (-T,\,T(T-1))\\
				(-1,0) & (T,\,-T(T+1)) & (-T^2,0) & (T,\,T(T+1)).
			\end{array}
		\end{align}
		Since $\mathcal{E}$ has full $\QQ(T)$-rational $2$-torsion, we can apply a result due to Gusi\'c and Tadi\'c (\cite[Theorem 1.1]{gusic}) to show that there are no other $\QQ(T)$-points. The non-constant square-free divisors $h(T)\in\ZZ[T]$ of $T^2$, $-(T^2-1)$, and $T^2(T^2-1)$ are all of the form
		\[h(T):=(-1)^{\epsilon_1}T^{\epsilon_2}(T+1)^{\epsilon_3}(T-1)^{\epsilon_4}\]
		for $\epsilon_i\in \{0,1\}$. For all sixteen of these polynomials $h(T)$, $h(6)$ is not a square in $\QQ$, and so the specialization map $\mathcal{E}(\QQ(T))\to E_6(\QQ)$ is injective. Since $E_6(\QQ)\cong \ZZ/2\ZZ\times \ZZ/4\ZZ$, we can conclude that $\mathcal{E}(\QQ(T))\cong\ZZ/2\ZZ\times \ZZ/4\ZZ$.
	\end{proof}

	In fact, any elliptic curve $E/\QQ$ with $\ZZ/2\ZZ\times\ZZ/4\ZZ\leq E(\QQ)$ satisfies $E\cong E_t$ for some $t$, as is shown in Chapter 4, Example (6.2) of~\cite{husemoller}. For any $t\in \QQ\setminus\{0,\pm 1\}$, we define $A_t\leq E_t(\QQ)$ to be the image of $\mathcal{E}(\QQ(T))$ under the specialization map; that is,
	\begin{align}\label{eq:Adef}
		\begin{aligned}
			A_t&:=\{O,\,(0,0),\,(-1,0),\,(-t^2,0),\,(-t,\,\pm t(t-1)),\,(t,\,\pm t(t+1))\}\\
			&\cong \ZZ/2\ZZ\times \ZZ/4\ZZ.
		\end{aligned}
	\end{align}

	\section{Rational right triangles}\label{sec:ratright}
	
	Fix $t\in\QQ_{>0}$, and let $\mathcal{S}_t$ denote the set of triples $(a,b,c)$ satisfying \cref{eq:righttri}. We would like to determine how the size of $\mathcal{S}_t$ depends on $t$. If $t=1$, an infinite collection of solutions is given by $a=\frac{1-r^2}{2r}$, $b=c=\frac{1+r^2}{2r}$ for $r\in\QQ\cap (0,1)$.
	
	When $t\neq 1$, solutions to \cref{eq:righttri} arise from rational points on $E_t$. Recall that $A_t$ was defined to be the isomorphic copy of $\ZZ/2\ZZ\times\ZZ/4\ZZ$ in $E_t(\QQ)$.

	\begin{lem}\label{prop:eighttoone}
		Fix $t\in\QQ_{>0}\setminus\{1\}$. The function
		\[\Phi:(x,y)\mapsto \left(\left|\frac{x^2-t^2}{2y}\right|,\,\left|\frac{x^2+2x+t^2}{2y}\right|,\,\left|\frac{x^2+2t^2x+t^2}{2y}\right|\right)\]
		is an eight-to-one map from $E_t(\QQ)-A_t$ to $\mathcal{S}_t$.
	\end{lem}
	\begin{proof}
		Let $\PP^2$ be the projective plane over $\QQ$ with coordinates $(r,s,w)$. Define a curve $C_t$ in $\PP^2$ by
		\begin{align}
			C_t:(w^2-r^2)(2s)=t(w^2-s^2)(2r).
		\end{align}
		There is an isomorphism $\Phi_1:E_t\to C_t$ defined by 
		\begin{align*}
			\Phi_1(x,y)&=(x+t^2,\,t(x+1),\,y),
		\end{align*}
		sending the identity of $E_t(\QQ)$ to $(0,0,1)$. The locus of points in $C_t$ with $rsw(w^2-r^2)=0$ consists of the eight points
		\begin{align*}
			\begin{array}{cccc}
				(0,0,1), & (-1,1,1), & (t,1,0), & (1,-1,1),\\
				(1,0,0), & (-1,-1,1), & (0,1,0), & (1,1,1),
			\end{array}
		\end{align*}
		which is exactly the image of $A_t$ under $\Phi_1$. Hence, if we let $C_t'$ be the affine open subset of $C_t$ defined by $rsw(w^2-r^2)\neq 0$, then $\Phi_1$ defines a bijection $E_t(\QQ)\setminus A_t\to C_t'(\QQ)$. 

		Now define a function $\Phi_2:C_t'(\QQ)\to \QQ^3$ by
		\begin{align}\label{eq:curvetotris}
		\Phi_2(r,s,w):=\left(\left|\frac{w^2-r^2}{2rw}\right|,\,\left|\frac{w^2+r^2}{2rw}\right|,\,t\left|\frac{w^2+s^2}{2sw}\right|\right).
		\end{align}
		If $(a,b,c)=\Phi_2(r,s,w)$, then $a,b,c>0$, $1+a^2=b^2$, and by the defining equation for $C_t$,
		\begin{align*}
			t^2+a^2&=t^2+\left(\frac{w^2-r^2}{2rw}\right)^2=t^2+t^2\left(\frac{w^2-s^2}{2sw}\right)^2=c^2.
		\end{align*}
		Thus the image of $\Phi_2$ lands in $\mathcal{S}_t$. Given any $(a,b,c)\in\mathcal{S}_t$, by the classical parametrization of Pythagorean triples there exist $r,s\in\QQ\setminus\{0,\pm 1\}$ such that 
		\begin{align*}
			(a,b)=\left(\frac{1-r^2}{2r},\,\frac{1+r^2}{2r}\right),\qquad (a,c)=\left(t\frac{1-s^2}{2s},\,t\frac{1+s^2}{2s}\right).
		\end{align*}
		Thus $(r,s,1)$ is a point in $C_t'(\QQ)$ mapping to $(a,b,c)$, so $\Phi_2:C_t'(\QQ)\to \mathcal{S}_t$ is surjective.
		
		Two points $(r,s,1),(r',s',1)\in C_t'(\QQ)$ map to the same element of $\mathcal{S}_t$ if and only if $r'\in\{r,-r,\frac{1}{r},-\frac{1}{r}\}$ and $s'\in\{s,-s,\frac{1}{s},-\frac{1}{s}\}$. The points $(r,s,1)$ and $(-r,s,1)$ are never both in $C_t'(\QQ)$, but there is a group $G$ of order $8$ generated by the involutions
		\[(r,s,1)\mapsto (-\tfrac1r,s,1),\qquad  (r,s,1)\mapsto (r,-\tfrac1s,1),\qquad  (r,s,1)\mapsto (-r,-s,1).\]
		acting freely on $C_t'$. Hence the fibers of $\Phi_2$ are exactly the orbits of $G$, so $\Phi_2$ is eight-to-one. 
		
		Finally, $\Phi:=\Phi_2\circ\Phi_1$ is eight-to-one because $\Phi_1:E_t(\QQ)\setminus A\to C_t'(\QQ)$ is a bijection and $\Phi_2:C_t'(\QQ)\to\mathcal{S}_t$ is eight-to-one. 
	\end{proof}

	\Cref{prop:eighttoone} shows that for $t\neq 1$, $\mathcal{S}_t$ is nonempty if and only if $E_t(\QQ)\setminus A_t$ is nonempty. If $E_t(\QQ)$ has positive rank, then $\mathcal{S}_t$ is infinite. It is also possible for $E_t(\QQ)$ to have a torsion point outside of $A_t$. The values of $t$ for which this occurs are classified by the following result. 

	\begin{lem}\label{prop:extrators}
		Let $t\in\QQ_{>0}\setminus \{1\}$. The following are equivalent:
		\begin{enumerate}[label=(\alph*)]
			\item The set $E_t(\QQ)_{\text{tors}}\setminus A_t$ is non-empty.
			\item $E_t(\QQ)_{\text{tors}}\cong \ZZ/2\ZZ\times \ZZ/8\ZZ$.
			\item There exists a solution $(a,b,c)\in\mathcal{S}_t$ such that $a(1,a,b)=(a,t,c)$ (i.e.\ the two right triangles are similar).
			\item $t$ and $t+1$ are both squares in $\QQ$.
			\item $t=\left(\frac{1-r^2}{2r}\right)^2$ for some $r\in\QQ\setminus \{0,\pm 1\}$.
		\end{enumerate}
	\end{lem}
	\begin{proof}
		\textbf{(a) implies (b).} By Mazur's classification of torsion subgroups of Mordell-Weil groups of elliptic curves over $\QQ$~\cite{mazur}, the only possible torsion subgroup strictly containing $\ZZ/2\ZZ\times\ZZ/4\ZZ$ is $\ZZ/2\ZZ\times\ZZ/8\ZZ$.
		\vspace{3pt}
		
		\noindent\textbf{(b) implies (c).} Suppose $E_t(\QQ)$ has a point $(u,v)$ of order $8$. Then $2(u,v)$ is a point of order $4$ on the identity component of $E_t(\RR)$, so without loss of generality we can assume $2(u,v)=(t,t(t+1))$. By the duplication law on $E_t(\QQ)$~\cite[Group Law Algorithm 2.3(c)]{silverman}, we can write
		\begin{align*}
			t&=\left(\frac{3u^2+2(1+t^2)u+t^2}{2v}\right)^2-(1+t^2)-2u\\
			&=\frac{(3u^2+2(1+t^2)u+t^2)^2-4u(u+1)(u+t^2)(2u+t^2+1)}{4u(u+1)(u+t^2)}\\
			&=\left(\frac{u^2-t^2}{2v}\right)^2.
		\end{align*}
		By \cref{prop:eighttoone}, we have $(a,b,c):=\Phi(u,v)\in\mathcal{S}_t$ with $a=\left|\frac{u^2-t^2}{2v}\right|$. Hence we have $t=a^2$, and
		\[c^2=t^2+a^2=a^2(a^2+1)=a^2b^2,\]
		so that $a(1,a,b)=(a,t,c)$.
		\vspace{3pt}
		
		\noindent\textbf{(c) implies (d).} Given a solution $(a,b,c)\in\mathcal{S}_t$ with $a(1,a,b)=(a,t,c)$, we have $a^2=t$, and therefore $b^2=1+t$.
		\vspace{3pt}
		
		\noindent\textbf{(d) implies (e).} Suppose $t=a^2$ and $1+t=b^2$ for some $a,b\in\QQ$. By a classical parametrization of Pythagorean triples, we must have $a=\frac{1-r^2}{2r}$ for some $r\in \QQ\setminus \{0,\pm 1\}$, and therefore $t=\left(\frac{1-r^2}{2r}\right)^2$.
		\vspace{3pt}
		
		\noindent\textbf{(e) implies (a).} If $t=\left(\frac{1-r^2}{2r}\right)^2$, then $E_t(\QQ)$ contains a point
		\[(u,v)=\left(\frac{(1-r)(1+r)^3}{4r^3},\,\frac{(1-r)(1+r)^3(1+r^2)(r^2-2r-1)}{16r^5}\right)\]
		of order 8.
	\end{proof}

	\Cref{prop:classify} is an immediate consequence of \cref{prop:eighttoone,prop:extrators}.
	
\section{Computing the root number}\label{sec:computeroot}

	\begin{lem}\label{lem:rootnum}
		Let $t\in\QQ\setminus \{0,\pm 1\}$. Set $w_2(t)=1$ if $\ord_2(t)=\pm 2$ or $\ord_2(t^2-1)=3$, and $w_2(t)=-1$ otherwise. The root number of $E_t$ (\cref{eq:Edef}) is given by
		\begin{align}\label{eq:rootnum}
			W(E_t)=-w_2(t)\left(\prod_{\substack{p\text{ odd}\\\ord_p(t)\neq 0}}(-1)\right)\left(\prod_{\substack{p\equiv 1\; \mathrm{mod} \; 4\\\ord_p(t^2-1)>0}}(-1)\right).
		\end{align}
	\end{lem}
	A version of this formula appears in~\cite[Proposition 11]{xarles}; we include a computation of this result here for completeness.
	
	The following results are relatively standard; we provide a reference for each. Note that we do not need a minimal model for $E$.
	\begin{lem}\label{lem:rootnumrefs}
		Let $E$ be an elliptic curve over $\QQ$ with an integral model.
		\begin{enumerate}[label=(\roman*)]
			\item $W_\infty(E)=-1$.~{\cite[Proposition 1]{rohrlich}} 
			\item If $E$ has good reduction at $p$, then $W_p(E)=1$.~{\cite[Proposition 2(iv)]{rohrlich}} 
			\item If $E$ has split multiplicative reduction at $p$ then $W_p(E)=-1$.~{\cite[Proposition 3(iii)]{rohrlich}}
			\item If $E$ has nonsplit multiplicative reduction at $p$ then $W_p(E)=1$.~{\cite[Proposition 3(iii)]{rohrlich}}
		\end{enumerate}
	\end{lem}

	Using these results we can compute the local root numbers $W_p(E_t)$ at all odd primes.
	
	\begin{lem}\label{lem:rootnumodd}
		If $p$ is odd, then the local root number of $E_t$ at $p$ is given by
		\[W_p(E_t)=\left\{\begin{array}{rl}
			-1, & \ord_p(t)\neq 0,\\
			-1, & p\equiv 1\pmod 4\text{ and }\ord_p(t^2-1)>0,\\
			1, & \text{otherwise.}
		\end{array}\right.\]
	\end{lem}
	\begin{proof}
		Let $t=\frac uv$ for some coprime integers $u,v$. The transformation $(x,y)\mapsto (v^2x,v^3y)$ sends $E_t$ to an integral Weierstrass model 
		\[E_{u,v}:y^2=x(x+u^2)(x+v^2)=x^3+(u^2+v^2)x^2+u^2v^2x.\]
		This model has discriminant $\Delta=16u^4v^4(u^2-v^2)^2$. The root number at $p$ can be computed using \cref{lem:rootnumrefs}. If $p\nmid \Delta$ then $E_{u,v}$ has good reduction at $p$, and so $W_p(E_t)=1$.
		
		If $p\mid uv$ (equivalently, if $\ord_p(t)\neq 0$), then $p\nmid u^2+v^2$, and so $(0,0)$ is a node on $\overline{E_{u,v}}$. Since $b_2=4(u^2+v^2)$ is a quadratic residue mod $p$, $E_{u,v}$ has split multiplicative reduction at $p$. Thus $W_p(E_t)=-1$.
		
		If $p\mid (u^2-v^2)$ (equivalently, if $\ord_p(t^2-1)> 0$), then the transformation $x\mapsto x+v^2$ takes $\overline{E_{u,v}}$ to the curve over $\FF_p$ given by
		\[y^2=(x-v^2)(x+u^2-v^2)x=x^2(x-v^2),\]
		on which $(0,0)$ is a node. If $p\equiv 1\pmod 4$, then $b_2=-4v^2$ is a square mod $p$, and so $E_{u,v}$ has split multiplicative reduction at $p$; thus $W_p(E_{u,v})=-1$. If $p\equiv 3\pmod 4$, then $b_2$ is not a square mod $p$, and so $E_{u,v}$ has nonsplit multiplicative reduction at $p$; thus $W_p(E_t)=1$.
	\end{proof}
	
	\Cref{lem:rootnumrefs} is not sufficient for computing the local root number at $2$. Here we use a table computed by Halberstadt~\cite{halberstadt}, which requires a Weierstrass model that is minimal at $2$; that is, a Weierstrass model which minimizes $\ord_2(\Delta)$. 
	
	\begin{lem}\label{lem:rootnum2}
		The local root number of $E_t$ at $2$ is given by
		\[W_2(E_t)=\left\{\begin{array}{rl}
			1, & \ord_2(t)=\pm 2\text{ or }\ord_2(t^2-1)=3,\\
			-1, & \text{otherwise.}
		\end{array}\right.\]
	\end{lem}
	\begin{proof}
		Let $t=\frac uv$ for some coprime integers $u,v$. A choice of minimal model for $E_t$ will depend on $\ord_2(uv)$, so we divide into two cases. We will see that in the first case, the only way to attain $W_2(E_t)=1$ is if $\ord_2(uv)=2$ (equivalently, $\ord_2(t)=\pm2$), and in the second case, the only way to attain $W_2(E_t)=1$ is if $\ord_2(u^2-v^2)=3$ (equivalently, $\ord_2(t^2-1)=3$).
		
		\vspace{3pt}
		
		\noindent
		\textbf{Case I.} Suppose $\ord_2(uv)\geq 2$. Without loss of generality, we can let $v$ be odd and $u=4w$ for $w\in\ZZ$. The transformation $(x,y)\mapsto (\frac14 v^2x,\frac18v^3(y-x))$ sends $E_t$ to an integral Weierstrass model
		\[E_{w,v}':y^2+vxy=x^3+4w^2x^2+w^2v^2x,\]
		with discriminant $\Delta=w^4v^4(4w+v)^2(4w-v)^2$. The local root number at $2$ can be computed using \cref{lem:rootnumrefs}. If $2\nmid w$, then $E_{w,v}'$ has good reduction at $2$, and so $W_2(E_t)=1$. If $2\mid w$, then $\overline{E_{w,v}'}$ has a singularity at $(0,0)$. Since $b_2=1$ is a quadratic residue mod $2$, $E_{w,v}'$ has split multiplicative reduction at $2$, and so $W_2(E_t)=-1$. 
		
		\vspace{3pt}
		
		\noindent
		\textbf{Case II.} Suppose $\ord_2(uv)\leq 1$. The transformation $(x,y)\mapsto (v^2x,v^3y)$ sends $E_t$ to an integral Weierstrass model 
		\[E_{u,v}:y^2=x(x+u^2)(x+v^2)=x^3+(u^2+v^2)x^2+u^2v^2x.\]
		Let us show that this is a minimal Weierstrass model at $2$. Suppose for the sake of contradiction that there exists a transformation $(x,y)\mapsto(\frac14(x-r),\frac18(y-s(x-r)-4t))$ taking $E_{u,v}$ to an integral Weierstrass model, for some $r,s,t\in\QQ$. The corresponding Weierstrass model then takes the form
		\begin{align*}
			y^2+sxy+t=x^3& + \frac1{4}(u^2+v^2+3r -s^2)x^2+
			\frac{1}{16} (3 r^2 - 8 s t + u^2 v^2 + 2 r (u^2 + v^2))x\\
			&+\frac1{64}(r^3 - 16 t^2 + r u^2 v^2 + r^2 (u^2 + v^2)).
		\end{align*}
		For this to be integral at $2$, each of $s,t,r$ must be $2$-adic integers. If $\ord_2(uv)=1$, then $u^2v^2\equiv 4\pmod 8$ and $u^2+v^2\equiv 1\pmod 4$, and there are no values of $r$ such that the coefficient of $x$ is an integer. So now consider the case $\ord_2(uv)=0$, so that $u^2v^2\equiv 1\pmod 8$ and $u^2+v^2\equiv 2\pmod 8$. From the coefficient of $x$, we conclude that $r$ must be odd. From the coefficient of $x^2$ we can therefore conclude that $s$ must be odd, and hence $r\equiv 1\pmod 4$. But this implies that the constant term is not in $\ZZ$. This shows that there is no isomorphic integral model with a smaller value of $\ord_2(\Delta)$, so $E_{u,v}$ is a minimal Weierstrass model at $2$.
		
		The model $E_{u,v}$ has invariants
		\begin{align*}
			c_4&=16(u^4-u^2v^2+v^4),\\
			c_6&=-32(2u^2-v^2)(u^2-2v^2)(u^2+v^2)\\
			\Delta&=16u^4v^4(u+v)^2(u-v)^2.
		\end{align*}
		Regardless of $u$ and $v$, exactly one of $2u^2-v^2$, $u^2-2v^2$, and $u^2+v^2$ is $2\pmod 4$, and the others are odd. Hence we can write $c_4=16c_4'$ and $c_6=64c_6'$, where $c_4'$ and $c_6'$ are odd integers.
		
		We now divide further into cases. By symmetry, the case $\ord_2(v)=1$ can be handled identically to the case $\ord_2(u)=1$, so without loss of generality we assume $v$ is odd. If $u$ and $v$ are both odd, then $u^2-v^2\equiv 0\pmod 8$.
		\begin{itemize}
			\item $\ord_2(u)=1$. Then $\ord_2(c_4)=4$, $\ord_2(c_6)=6$, and $\ord_2(\Delta)=8$. If we set $u=4n+2$ and $v=2m+1$, then 
			\[2c_6'+c_4'\equiv 16m^2(m^2+1)+7\equiv 7\pmod{32}.\]
			By~\cite[Table 1]{halberstadt}, $W_2(E_t)=-1$.
			
			\item $\ord_2(u)=0$ and $\ord_2(u^2-v^2)=3$. Then $\ord_2(c_4)=4$, $\ord_2(c_6)=6$, and $\ord_2(\Delta)=10$. We have $2u^2-v^2\equiv 1\pmod 4$, $u^2-2v^2\equiv 3\pmod 4$, and $u^2+v^2\equiv 2\pmod 8$, so that $c_6'\equiv 1\pmod 4$. So by~\cite[Table 1]{halberstadt}, $W_2(E_t)=1$.
			
			\item $\ord_2(u)=0$ and $\ord_2(u^2-v^2)=4$. Then $\ord_2(c_4)=4$, $\ord_2(c_6)=6$, and $\ord_2(\Delta)=12$, and as above, $c_6'\equiv 1\pmod 4$. So by~\cite[Table 1]{halberstadt}, $W_2(E_t)=-1$.
			
			\item $\ord_2(u)=0$ and $\ord_2(u^2-v^2)>4$. Then $\ord_2(\Delta)>12$, and so $\ord_2(j)<0$, where $j=\frac{c_4^3}{\Delta}$. The transformation $x\mapsto x+v^2$ takes $\overline{E_{u,v}}$ to the curve over $\FF_p$ given by $y^2=x^2(x-v^2)$, which has a singularity at $(0,0)$ and $b_2=-4v^2\equiv 0\pmod 2$, so $E_{u,v}$ has additive reduction at $2$. Since $E_{u,v}$ has additive reduction, $\ord_2(j)<0$, and $c_6'\equiv 1\pmod 4$, we have $W_2(E_t)=-1$~\cite[Section 3]{connell}.\qedhere
		\end{itemize}		
	\end{proof}
	\Cref{eq:rootnum} follows by plugging in the results of \cref{lem:rootnumrefs}(i), \cref{lem:rootnumodd}, and \cref{lem:rootnum2} into the formula
	\[W(E_t)=W_\infty(E_t)\prod_{p}W_p(E_t).\]

	\subsection{Proof of \cref{thm:rootnumtotest}}
		We first show that the set 
		\[\{t\in\QQ\setminus \{0,\pm1\}:W(E_t)=-1\}\]
		has density $\frac12$. We proceed as in~\cite{helfgott}. For each place $v$ of $\QQ(T)$, we obtain a homogeneous polynomial $P_v\in\ZZ[x,y]$: this equals $x-ty$ if $v$ corresponds to a degree one place $T=t$, and $y$ if $v$ is the place at infinity. The elliptic curve $\mathcal{E}/\QQ(T)$ has bad reduction at $t\in\{0,1,-1,\infty\}$, and these are all places of multiplicative reduction. The valuation of $\Delta(\mathcal{E})$ at each of these places is $4$, $2$, $2$, and $4$, respectively; since the valuation is never $6\mod 12$, every quadratic twist of $\mathcal{E}$ also has bad reduction at these places as well (in which case we say $\mathcal{E}$ has ``quite bad reduction'' at these places)~\cite[6.2.1]{helfgott}. Hence we have
		\[M_{\mathcal{E}}:=\prod_{\mathcal{E}\text{ mult.\ red.\ at }v}P_v=xy(x+y)(x-y)=\prod_{\mathcal{E}\text{ quite bad red.\ at }v}P_v=:B_{\mathcal{E}}.\]
		By the discussion in~\cite[Section 2.2]{helfgott}, hypotheses $\mathscr{A}_2(B_{\mathcal{E}})$ and $\mathscr{B}_2(M_{\mathcal{E}})$ are both satisfied unconditionally (the former because $B_{\mathcal{E}}$ has no irreducible factor of degree larger than $6$, and the latter because $M_{\mathcal{E}}$ is a product of four distinct linear factors). Therefore, since $M_{\mathcal{E}}$ is not constant, the map $t\mapsto W(E_t)$ has zero average over the rationals~\cite[Main Theorem 2]{helfgott}. This implies that the set of rationals with $W(E_t)=-1$ has density $\frac12$ as desired.
		
		For $t\in\QQ\setminus \{0,\pm1\}$, the set $\mathcal{P}_t$ is exactly the set of primes contributing a $-1$ to the formula for $W(E_t)$ in \cref{lem:rootnum} (note that it is impossible to have both $\ord_p(t)\neq 0$ and $\ord_p(t^2-1)>0$, so each prime contributes at most once to \cref{lem:rootnum}). Hence
		\[W(E_t)=-(-1)^{|\mathcal{P}_t|}.\]
		Thus $W_p(E_t)=-1$ if and only if $|\mathcal{P}_t|$ is even. 
		
		Assuming \cref{conj:parity}, if $|\mathcal{P}_t|$ is even then $E_t(\QQ)$ has positive rank. Now let 
		\[S=\{t\in\QQ\setminus \{0,\pm1\}:\rank E_t(\QQ)\geq 2\}\cup\{-1,0,1\}.\] 
		The rank of $\mathcal{E}(\QQ(T))$ is $0$ by \cref{lem:MWgroup}, so assuming \cref{conj:density}, $S$ has density $0$. 
		The curves $E_0$, $E_{-1}$, and $E_1$ each have infinitely many rational points, and for all $t\in\QQ\setminus \{0,\pm1\}$, if $E_t(\QQ)$ is infinite, then either $\rank E_t(\QQ)\geq 2$ or $W(E_t)=-1$ (again assuming \cref{conj:parity}). 
		
	\printbibliography

\end{document}